\documentclass[a4paper,11pt]{amsart}
\usepackage{amssymb}
\usepackage{graphicx}
\usepackage{amscd}
\usepackage{epstopdf}

\newcommand{\comm}[1]{}

\newcommand{\eps}{\varepsilon}
\newcommand{\hide}[1]{}
\renewcommand{\marginpar}[1]{}

\def\C{{\mathbb C}}

\def\SS{{\mathcal S}}

\def\al{\alpha}

\newtheorem{theorem}{Theorem}
\newtheorem{lemma}[theorem]{Lemma}

\newtheorem{corollary}[theorem]{Corollary}

\theoremstyle{definition}
\newtheorem{definition}[theorem]{Definition}
\theoremstyle{remark}
\newtheorem{remark}{Remark}

\begin{document}

\title[On the speed of convergence of Newton's method]{On the speed
of convergence of Newton's method for complex polynomials}
\author{Todor Bilarev}
\address{
Institut f\"{u}r Mathematik, Humboldt-Universit\"{a}t zu Berlin, Unter den Linden 6, D-10099 Berlin, Germany}
\email{bilarev.todor@gmail.com}

\author{Magnus Aspenberg}
\address{Lund University,
Centre of Mathematical Sciences, Box 11, 221 00 Lund, Sweden
}
\email{maspenberg@gmail.com}

\author{Dierk Schleicher}
\address{Research I, Jacobs University Bremen, Postfach 750 561, D-28725
Bremen, Germany}
\email{dierk@jacobs-university.de}
\begin{abstract}
       We investigate Newton's method for complex polynomials of
arbitrary degree $d$, normalized so that all their roots are in the
unit disk. For each degree $d$, we give an explicit set
$\mathcal{S}_d$ of $3.33d\log^2 d(1 + o(1))$ points with the
following universal property: for every normalized polynomial of
degree $d$ there are $d$ starting points in $\mathcal{S}_d$ whose
Newton iterations find all the roots with a low number of iterations: if the roots are uniformly and
independently distributed, we show that with probability at least $1-2/d$ the number of iterations for
these $d$ starting points to reach all roots with precision
$\varepsilon$ is  $O(d^2\log^4 d + d\log|\log \varepsilon|)$. This is
an improvement of an earlier result in \cite{D-draft}, where the
number of iterations is shown to be
        $O(d^4\log^2 d + d^3\log^2d|\log \varepsilon|)$ in the worst case
(allowing multiple roots) and
        $O(d^3\log^2 d(\log d + \log \delta) + d\log|\log \varepsilon|)$
for well-separated (so-called $\delta$-separated) roots.

        Our result is almost optimal for this kind of starting points in
the sense that the number
of iterations can never be smaller than $O(d^2)$ for fixed $\eps$.
\end{abstract}
      \maketitle

\section{Introduction}

Newton's root finding method is an old and classical method for
finding roots of a differentiable function; it goes back to Newton in
the 18th century, perhaps earlier. It was one of the main reasons why
A.\ Douady, J.\ Hubbard and others in the late 1970s studied
iterations of complex analytic functions. The main question was to
know where to start the Newton iterative method in order to converge
to the roots of the function. Newton's method is known as rapidly
converging near the roots (usually with quadratic convergence), but
had a reputation that its global dynamics was difficult to
understand, so that in practice often other methods for root finding are
used. See \cite{JohannesNewtonSurvey,D-Fields} for an overview on recent
results about Newton's method. In the following work we will be concerned with the problem of finding the roots of complex polynomials.

Meanwhile, some small sets of good starting points are known: there
are explicit deterministic sets with $O(d\log^2d)$ points that are
guaranteed to find all roots of appropriately normalized complex polynomials
of degree $d$ \cite{HSS}, and probabilistic sets with as few as
$O(d(\log\log d)^2)$ points \cite{BLS}.

We are interested in the question how many iterations are required
until all roots are found with prescribed precision $\eps$. In
\cite{D-draft}, it is shown that among a set of starting points as
specified above, there are $d$ points that converge to the $d$ roots
and require at most $O(d^4\log^2d+d^3\log^2d|\log\eps|)$ iterations to get
$\eps$-close to the $d$ roots in the worst case; for randomly placed
roots (or for roots at mutual distance at least $\delta$ for some
$\delta>0$), the required number of iterations is no more than
$O(d^3\log^3d+d\log|\log\eps|)$ (with the constant depending on
$\delta$). This is about one power of $d$ away from the best possible
bounds.

In this paper, we show that Newton's method is about as fast as
theoretically possible. We consider the space of polynomials of
degree $d$, normalized so as to have all roots in the complex unit
disk $\mathbb D$. Our main result is the following.

\begin{theorem}[Quadratic Convergence in Expected Case] \label{Thm:Main}
For every degree $d$, there is an explicit universal set $\SS_d$ of
points in $\C$, with $|\SS_d|=3.33d\log^2d(1+o(1))$, with the
following property:
suppose that $\al_1, \ldots, \al_d$ are uniformly and independently
distributed in the unit disk and consider the polynomial $\prod_{j=1}^d (z-\al_j)$.
Then with probability at least $1-2/d$ there are $d$ starting points in $\SS_d$ such that the number of
iterations needed to approximate all $d$ roots with any given precision $\eps > 0$
starting at these $d$ points is at most 
\[ 
	C(d^2\log^4 d + d\log|\log \varepsilon|)
\]
for a universal constant $C$.
\end{theorem}

\begin{remark}
As stated, the theorem deals with $d$ distinguishable (i.e., ordered)
roots and their associated probability distribution.
We prove that the same result holds if we identify our polynomials in terms
of their sets of \emph{indistinguishable} roots, as two polynomials
$\prod_{j=1}^d (z-\al_j)$ and $\prod_{j=1}^d (z-\beta_j)$
are the same if their unordered sets of roots $\{\al_1,\ldots, \al_d\}$ and
$\{\beta_1,\ldots, \beta_d\}$ are equal (of course taking multiplicities into account).
\end{remark}
\begin{remark}
This bound on the number of iterations is optimal in the sense that
there is no bound on the number of iterations in the same generality
that for fixed $\eps$ has  asymptotics in $o(d^2)$ (see Remark~\ref{Rem:Optimal}), so we are away
from the best possible bound only by a factor of about $O(\log ^4d)$.
\end{remark}

%\newpage

\begin{remark}[Are We Proving the Right Theorem?]
One might ask how useful a result is that deals with finding the roots of a polynomial with given factorization. We would like to point out that for our analysis we do not assume that the roots are known: all we assume is that the degree is known and that the roots satisfy a certain bound (they are contained in the unit disk), which is easily achieved by rescaling. In this paper, we use on the space of polynomials  the probability measure $\mu_{\text{roots}}$ induced by the Lebesgue measure of the positions of the roots. One could equally well use various other measures on the space of normalized polynomials of given degree $d$, such as the measure $\mu_{\text{coeffs}}$ induced by Lebesgue measure on the complex coefficients of the polynomials, or possibly other measures depending on different contexts that specify the polynomials. 

For instance, in our applications we are often interested in the periodic points of some period $N$ of a given polynomial $q$ because these approximate the measure of maximal entropy on the Julia set in the dynamical plane of $q$ \cite{Lyubich}. These periodic points are roots of $p(z)=q^{\circ N}(z)-z$, but it may be neither desirable nor even feasible to compute the coefficients of $p$, and neither is it necessary to know them in order to find the roots. If $q$ has degree $d$, then evaluating $p$ by iteration takes $O(dN)$ computations, while even finding the coefficients of $p$ requires $O(d^N)$ computations.
Simple example: there are $2^N=1024$ periodic points of period $N=10$ of $q(z)=z^2+2$, and all satisfy $|z|\le 2$, while the constant coefficient of $q^{\circ N}(z)-z$ has magnitude about $2^{2^{10}}> 10^{300}$. Finding all $2^N=1024$ periodic points with our methods is very well feasible, but certainly \emph{not} by expanding $q^{\circ N}$ in coefficients! 

We see no reason why one measure on the space of polynomials would be more fundamental than any other, so we use the one that is most suitable in our context, and that is the measure $\mu_{\text{roots}}$. Of course, there is a well known map $F\colon \{\alpha_1,\dots,\alpha_d\}\mapsto (\text{coefficients})$ that maps roots to coefficients, and it gives accordingly a measure $F_*(\mu_{\text{roots}})$ on the space of coefficients, different from $\mu_\text{coeffs}$. 
The probabilistic nature of our results is essentially an estimate on the measure of the set of those polynomials where we cannot assure fast convergence of Newton's method (the set of ``bad polynomials'') with respect to $\mu_\text{roots}$; this is the set where polynomials have multiple or near-multiple roots. 

If one wants to estimate the set the ``bad polynomials'' with respect to $\mu_{\text{coeffs}}$, 
one needs to work out the exact bounds of the measure transformation $F_*$, and this would be an exercise on a well known algebraic map that is besides the point of the present paper. It might be worth pointing out that $F$ is a polynomial map that has all its singularities on the locus of multiple roots (everywhere else, it has a local algebraic inverse), so $F$ is especially contracting near the locus of multiple roots: in fact, the Jacobian of $F$ is just the absolute value of the Vandermonde determinant. 
 It thus is reasonable to believe that the probability in $\mu_{\text{coeffs}}$ of the set of ``bad polynomials'' is much smaller than in $\mu_\text{roots}$, so one could expect even better bounds with respect to the former probability distribution. The corresponding exercise is not the focus of the present paper and would distract from our main contributions.
\marginpar{add results on Burrus etc}
\end{remark}

\begin{remark}
This paper is a result of the bachelor thesis of the first named author at Jacobs University Bremen.
\end{remark}

\section{Good starting points for Newton's method}
\label{sec: prelim results}

Studying the geometry of the immediate basins outside the unit disk
$\mathbb{D}$,  in \cite{HSS} we proved the existence of a universal
starting set with $1.11 d\log^2 d$ points depending only on $d$ such
that for every polynomial of degree $d$ with all roots in the unit
disk, and for every root, there is a point in the set which is in the
immediate basin of this root. Enlarging the set by a factor of 3
approximately, in \cite{D-draft} we obtained a set of starting points
$\mathcal{S}_d$ which ensured that for each polynomial $p$ and each
root $\alpha$ there is a point $z$ in $\mathcal{S}_d$ intersecting
the immediate basin $U$ of $\alpha$ in the ``middle third`` of the
``thickest'' \textit{channel}, where a channel is an unbounded
connected component of $U \setminus \overline{\mathbb{D}}$. Being in
this middle third implies an
upper bound on the displacement $d_U(z, N_p(z))$ in terms of the
Poincar\'e metric of the immediate basin. We can guarantee that the
orbit of $z$ under iteration of the Newton map does not leave $D_R(0)$,
the disk of radius $R$ centered at the origin for some bounded value
of $R$; moreover, the hyperbolic geodesics within $U$ connecting any point on the orbit of $z$ to the next orbit point is also contained in the same $D_R(0)$.
We will refer to such points $z$ as having
\emph{$R$-central orbits}. 

More precisely, let $\mathcal{S}_d$ be defined as follows.
\begin{definition}[Efficient Grid of Starting Points]

For each degree $d$, construct a circular grid $\mathcal{S}_d$ as
follows. For $k = 1, \ldots, s = \lceil 0.4\log d\rceil$, set
$$r_k = (1+\sqrt{2})\left(\frac{d-1}{d}\right)^\frac{2k-1}{4s},$$
and for each circle around 0 of radius $r_k$, choose $\lceil
8.33d\log d\rceil$ equidistant points (independently for all the
circles).

\end{definition}

The set $\mathcal{S}_d$ thus constructed has $3.33(1+o(1))d\log^2 d$
points. The following theorem is proven in \cite[Theorem~8]{D-draft}.

\begin{theorem}\label{thm: set of starting points}
     For each degree $d$, the set $\mathcal{S}_d$ has the following
universal property. If $p$ is any complex polynomial, normalized so
that all its roots are in $\mathbb{D}$, then there are $d$ points
$z^{(1)}, \ldots, z^{(d)}$ in $\mathcal{S}_d$ whose Newton iterations
converge to the $d$ roots of $p$. If $\alpha$ is a root of $p$ and
$U$ is the immediate basin of $\alpha$, then there is an index $i$
such that $z^{(i)}\in U$ with $d_U(z^{(i)}, N_p(z^{(i)}))< 2\log
d$. In addition,  $z^{(1)}, \ldots, z^{(d)}$ have $R$-central orbits
for
      $$R \leq 5\left(\frac{d}{d-1}\right)^{\lceil 5\pi(1+\log d)\rceil}
< 5 \exp \left( \frac{\left\lceil 5\pi(1+\log d) \right\rceil}{d} \right) .$$
\end{theorem}
For $d = 100$, we have $R < 12.3;$ for $d = 1000$, we have $R < 5.7$;
and asymptotically the upper bound on $R$ tends to $5$.

     The result provides an upper bound for $R$ that is uniform in $d$.
This set of starting points will be the basis for the discussion
which follows.

      \section{Uniformly distributed roots}\label{sec: result}
In this manuscript we investigate the Newton map for complex
polynomials with randomly distributed roots. In this section, we fix
notation and give the strategy of the
proof of our main result, Theorem~\ref{Thm:Main}.

Let $\alpha$ be a simple root
of a polynomial $p(z)=(z-\alpha_1)(z-\alpha_2)\cdot\dots\cdot(z-\alpha_d)$ of degree $d$ and $U$ be the immediate basin of
attraction of $\alpha$. The associated Newton map is $N_p(z)=z-p(z)/p'(z)$.
By the discussion in the
previous section, there exists $z_1\in \mathcal{S}_d$ with
$R$-central orbit in $U$; in particular, under iteration of the Newton map
$N_p$ the orbit converges to $\alpha$ and stays within $D_R(0)$. Let
$z_{n+1} := N_p(z_{n})$ for $n\geq 1$. For any two
consecutive points $z_n$ and $z_{n+1}$ along the orbit of $z_1$,  in
\cite[Section~4]{D-draft} we constructed ``thick'' curves that connect these two points and that,
roughly speaking, ``use up'' area at least $|z_n -
z_{n+1}|^2/(2\tau)$ with $\tau := d_U(z_1, z_2) < 2\log d$. These ``thick'' curves will be described in Lemma~\ref{lemma: area bound}; basically, they
are certain neighborhoods (with respect to the Euclidean metric of $\C$) of the hyperbolic geodesic in $U$ connecting $z_n$ to $z_{n+1}$; the main task in \cite{D-draft} is to show that these Euclidean neighborhoods are contained in $U$ and that they are ``almost'' disjoint in the sense that each element of area is accounted for no more than $2\tau+6$ times \cite[Lemma~11]{D-draft}. The available Euclidean area within $U\cap D_R(0)$ thus bounds the possible number of iteration steps for which $|z_n-z_{n+1}|$ satisfies a certain lower bound; conversely, when $|z_n-z_{n+1}|$ is small, then we are in the domain of quadratic convergence (near the root $\alpha)$ and only $O(\log|\log\eps|)$ further iterations are required.

The bound $O(d^3\log^3d +d\log|\log \varepsilon|)$ on the number of iterations in \cite{D-draft}
follows from lower bounds on the
displacements $|z_n - z_{n+1}|$ along the orbit. The main improvement
in this paper
is on the lower bounds on the displacements when the roots are
randomly distributed.

As in \cite{D-draft}, we partition $D_{R}(0)$ (the disk of radius $R$
centered at the origin) into domains
\[
S_k := \left\{z\in D_{R}(0): \min_j|z-\alpha_j| \in \left(2^{-(k+1)},
2^{-k}\right]\right\}, \ k \in \mathbb{Z}
\;.
\]
It turns out that if the roots are randomly distributed in the unit
disk, then with high probability (at least $1-1/d$) the following holds (see Corollary~\ref{cor: displacement}): there exists a
universal constant $C >0$ such that for every $n$ we have the following
estimates 
\[ |z_n - z_{n+1}| \ge \left\{ \begin{array}{ll}
            \frac{C}{d\log d} & \mbox{if $z_n\in S_k$ with $k
\leq \log_2 d$};\\
           \frac{C}{2^k k } & \mbox{otherwise}.\end{array} \right.
\]
If $z_n\in S_k$ with $k\le \log_2d$, then we say that we are ``in the
far case'', as $z_n$ is far from all the roots. Since each such
iteration ``uses up'' an area of at least $|z_n-z_{n+1}|^2/(2\tau)$, 
and at least one in $2\tau+6$ such regions contributes with its full area to the total space required, the total number of orbit
points in the far case is bounded by $O(d^2\log^4d)$.

On the other hand,
\cite[Lemma~16]{D-draft} says that if the orbit gets very close to
some root in comparison to the other roots, then it has entered the
region of quadratic convergence of that root where only
$\log_2|\log_2 \varepsilon - 5|$ iterations are sufficient to approximate it
within an $\varepsilon$-neighborhood. We call this ``the near case''.

For randomly distributed roots, the mutual distance between roots is
large enough so away from the region of quadratic convergence, we
only need to consider essentially $k \le 3+5/2\log_2 d$. We define the ``intermediate case'' as those $z_n\in S_k$
with $\log_2d< k\le 3+5/2\log_2 d$. Each domain $S_k$ has area
$O(d 4^{-k})$ and each iteration with $z_n\in S_k$ uses area about
$(C/2^kk)^2/2\tau\approx C^2/4^kk^2\tau$, the number of orbit points
in the intermediate case is at most $O(dk^2\tau)$ for each $k$, times
the usual factor $2\tau+6$ to make the areas disjoint. But
$\log_2d<k\leq 3+5/2\log_2 d$ and $\tau=O(\log d)$, so the total
number of iterations in the intermediate case is $O(d\log^5d)$; this is dominated by the ``far case''.

In the subsequent sections we will make these arguments precise.

\subsection{Distribution of the roots}\label{sec: root distribution}
In order to get a lower bound on the expected displacement, we will
first investigate the distribution of the roots. We will be
interested in two different kinds of probability spaces. The first
space $\mathcal{P}_d =\{(x_1, \ldots, x_d) : x_i\in \mathbb{D}\}$
consists of all polynomials with $d$ \textit{distinguishable} roots
in the unit disk, normalized so as to have leading coefficients $1$,
and the probability measure is induced by Lebesgue
measure on $\mathbb{D}^d$. The second space $\mathcal P_d/\Sigma_d$
consists of all polynomials with \textit{indistinguishable} roots in
the unit disk, i.e.\ the quotient probability space of the standard
action of the symmetric group $\Sigma_d$ on $\mathcal{P}_d$ defined by
permuting the roots.

The following lemma is the probabilistic ingredient of the main
theorem. It certainly isn't new, but easier verified than
looked up in the library.

\begin{lemma}[Base-$d$ numbers]\label{lemma: numbers}
     Let $M_d$ be the set of all $d$-digit numbers in base $d$. There exists a universal constant $C >0$ such that for each $d\in \mathbb{N}$ we have:

(a) The probability that a randomly chosen number $a\in M_d$
does not have a digit repeating more than $C \log d$ times is at
least $1-1/d$.

     (b) Let $\sim$ be the equivalence relation on $M_d$ defined as
follows: $a \sim b \Leftrightarrow \exists \sigma \in \Sigma_d$ with $a =
\sigma b$, i.e.\ two elements are equivalent if they have the same
sets of digits counted with multiplicities. Then the probability that
a randomly chosen element $[a]\in M_d/\sim$ does not have a
digit repeating more than $C \log d$ is at least $1-1/d$.

\end{lemma}
\begin{proof}
(a)  For fixed $i$, the number of $d$-digit numbers which contain at
least $m$ digits $i$ is at most
$\binom{d}{m}d^{d-m}$. Thus the number of $d$-digit numbers
which contain a symbol repeating at least $m$ times is at most
      $$d \binom{d}{m}d^{d-m} < \frac{d}{m !} d^{d}.$$
      So the probability that a randomly selected number in $M_d$
contains at least $m$ identical digits is at most $\frac{d}{m
!}$ since $|M_d| = d^d$. Therefore, with probability at least
$1-\frac{d}{m!}$, a randomly selected number in $M_d$ does not
have a digit repeating more than $m$ times.

      Note that if $m!\ge d^2$ we have $1-\frac{d}{m!} \geq 1 -
\frac{1}{d}$. Therefore, by taking $m$ such that $(m - 1)!
< d^2 \leq m !$ (which implies that $m$ is of magnitude  $O(\log d)$), we
prove the first part of the claim.

     (b)  Note that the elements of $ M_d/\sim$ can be mapped bijectively
to the set $\hbox{Mult}_d = \{(x_0, \ldots, x_{d-1}) : x_i \in
\mathbb{Z}_{\geq 0}, x_0 + \ldots + x_{d-1} = d\}$ as follows: for
$[a]\in M_d/\sim$ let $x_i$ be the multiplicity of digit $i$
in every $a\in [a]$. It is well known and easy to see that

\begin{equation}\label{eq: multiset coefficient}
     \left|\left\{\rule{0pt}{10pt}(x_0, \ldots, x_{r-1}) : x_i\in
\mathbb{Z}_{\geq 0}, x_0 + x_1 + \ldots + x_{r-1} = n\right\}\right|
= \binom{n + r - 1}{r - 1}
\;.
\end{equation}
Thus we have $|\hbox{Mult}_d| = \binom{2d - 1}{d-1}$. On the
other hand, the number of elements in $\hbox{Mult}_d$ with first
component at least $m$ is equal to the cardinality of
\[
\left\{(x_1, \ldots, x_{d-1}) : x_i \in \mathbb{Z}_{\geq 0}, x_1 + \ldots
+ x_{d-1} \leq d-m \right\}
\]
which has the same cardinality as
\[
\{(y_0, x_1, \ldots, x_{d-1}) : y_0, x_i \in \mathbb{Z}_{\geq 0}, y_0 +
x_1 + \ldots + x_{d-1} =  d-m \}
\;.
\]
Again by (\ref{eq: multiset coefficient}) this quantity equals to
$\binom{2d - m - 1}{d-1}$. Therefore, the number of elements of
$\hbox{Mult}_d$ with a component at least $m$, i.e.\ the number
of elements of $M_d/\sim$ with a digit repeating at least $m$
times, is at most $d\binom{2d - m - 1}{d-1}\;.$
Hence the probability that a number of $M_d/\sim$ has a digit
repeating at least $m$ times is at most
\begin{align*}
d\frac{\binom{2d - m - 1}{d-1}}{\binom{2d - 1}{d-1}} 
&= d\frac{
(2d - m - 1)! (d-1)! d!}{(2d-1)! (d-1)! (d - m)!} = 
\hide{ 
\\ &= d\frac{(d-m + 1)(d-m + 2)\cdots d}{(2d - m)(2d -
m + 1)\cdots (2d - 1)} \leq d \left(\frac{1}{2}\right)^{m -
1}\frac{d}{2d - 1}.
}
\\ 
&= d\frac{d (d-1) \cdots (d-m + 2)(d-m + 1)}{(2d-1)(2d-2)\cdots (2d -
m + 1)(2d - m)} \leq d \left(\frac{1}{2}\right)^{m -
1}\frac{d}{2d - 1}.
\end{align*}
 Hence for $m = \lceil 2\log_2 d + 1
\rceil$ (which is of magnitude $O(\log d)$) the second part of the claim follows.
\end{proof}

If the roots are randomly distributed in the unit disk one should
expect that the number of roots in a region is proportional to its
area. The previous claim easily implies the following statement.

\begin{lemma}\label{lemma:disk}

There exists a universal constant $C >0$ such that, for each degree $d\in \mathbb{N}$, the following holds with probability at least $1-1/d$: if a polynomial with roots $\alpha_1, \ldots, \alpha_d$ is randomly chosen in $\mathcal{P}_d$ or $\mathcal{P}_d/ \Sigma_d$, then every disk in $\C$ of area $A$ contains not more than $k(A)$ points among $\alpha_1, \ldots, \alpha_d$, where
\[
k(A) = \left\{ \begin{array}{ll}
              C d \log d\cdot A & \mbox{if $A \geq 1/d$}; \\
            C\log d& \mbox{otherwise}.\end{array} \right.
\]
\end{lemma}
\begin{proof}
We first argue that it suffices to prove the claim in the case when $d$ is an odd square. If not, then let $D$ be the smallest odd square bigger than $d$; adjoining $D-d$ additional random roots to a new polynomial of degree $D$, the claim holds with probability $1-1/D\ge 1-1/d$. In the process, the constant $C$ may change by no more than $D\log D/d\log d$, but this factor is bounded (and tends to $1$ as $d\to\infty$). 

We may thus assume that $d = (2k+1)^2$ for an integer $k$.
Then the unit disk can be subdivided into $d$ pieces
as follows (compare Fig.~\ref{fig: circle partitioned}): the first
piece is a disk with center $0$ and radius $r_0 = 1/\sqrt{d};$ next,
consider the annuli $A_s$ bounded between circles around $0$ of radii
$(2s-1)r_0$ and $(2s+1)r_0$ for $s = 1, \ldots k$, and subdivide each
annulus $A_s$ into exactly $8s$ pieces of equal area by drawing $8s$
radial segments. Thus we construct exactly $d$ pieces with equal area
and diameters comparable with $r_0$. Now, let a polynomial with roots $\alpha_1, \ldots, \alpha_d$ be randomly chosen in $\mathcal{P}_d$ or $\mathcal{P}_d/\Sigma_d$. By Lemma \ref{lemma: numbers} it
follows that each of the pieces contains at most $O(\log d)$ 
of the
points $\alpha_1, \ldots, \alpha_d$ with probability at least $1-1/d$ (in both
cases of distinguishable and indistinguishable roots): in the case of
distinguishable roots, the $i$-th digit of a $d$-digit number
specifies the number of the piece containing the $i$-th root; in the
other case, the same symmetries apply on both sides of the equality.

\begin{figure}
      \begin{center}
      \framebox{\includegraphics[width=0.5\textwidth]{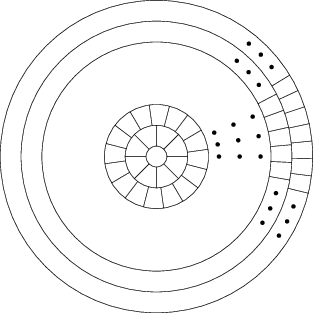}}
      \caption{Partition of the unit disk into smaller pieces of comparable sizes.\label{fig: circle partitioned}}
      \end{center}
\end{figure}

Hence, the claim is true for that particular partition of the unit
disk. This implies the general claim as follows. It is easy
to see that each square of side length at most $r_0$ in the complex
plane can intersect at most a constant number $C_1$ of these pieces,
where $C_1$ does not depend on $d$. Consider
a square $S$ for which the unit circle is inscribed, for example the
one with sides parallel to the real and imaginary axes. Subdivide it
into $d$ equal squares of
side length $r_0 $ (using the fact that $d$ is a square of an integer and
$r_0=1/\sqrt d$). Then each of these smaller
squares will intersect at most $C_1$ pieces from the partition of the
unit disk (some squares will not intersect any). Therefore, each of
the small squares contains at most $C_2\log d$ points for some
constant $C_2$ that does not depend on $d$. Since each square of
side length $r_0$ (possibly rotated) intersects at most $9$ of these
squares dividing $S$, we conclude that each square of side length
$r_0$ contains, with probability at least $1-1/d$,
at most $C_3\log d$ of the points $\alpha_1,\ldots, \alpha_d$ for $C_3 = 9 C_2$. If
we group every 4 neighboring small squares (of side length $r_0$) and
repeat the argument, we get that each square of side length $2r_0$
contains at most $4 C_3\log d$ points, and so on for squares of side
length $4r_0, 8r_0,\ldots $. Thus an arbitrary square of side length
$x\in[2^kr_0,2^{k+1}r_0]$ with $k\geq 0$ contains at most $2^{2k + 2} C_3\log d\le 4C_3 (x^2/r_0^2)\log d \approx 4C_3x^2d\log d$ points since it is contained in some square of side length $2^{k+1}r_0$. Thus, by enlarging the
constant by a factor of 4, the lemma will hold true for squares. Since
each disk of radius $r$ is contained in a square of side length $2r$,
the bound on the number of points in an arbitrary disk follows.
\end{proof}

Now we prove the following claim about the mutual distance for
randomly distributed points in the unit disk; we will use this in Lemma~\ref{lem: quadratic convergence} to derive sufficient conditions for fast convergence. The next few lemmas will use a parameter $\eta>0$ that we want to keep flexible for now, until we fix a choice in Theorem~\ref{thm:one root}.

\begin{lemma}\label{lemma: mutual distance}
      Let the polynomial $p$ be randomly chosen in $\mathcal{P}_d$ or $\mathcal{P}_d/\Sigma_d$. Then for any $\eta > 0$ the mutual distance between any pair of its roots is at least ${1}/{d^{1+\eta}}$ with probability at least $1-1/d^{2\eta}$.
\end{lemma}
\begin{proof}
      First, note that the claim for a randomly chosen polynomial in
$\mathcal{P}_d/\Sigma_d$ is equivalent to the claim for a randomly chosen
polynomial in $\mathcal{P}_d$ (the symmetry group $\Sigma_d$ acts in the same way on the set of all roots both on the space of all polynomials, and on the subspace of those where the distance between roots is at least $1/d^{1+\eta}$; polynomials with exact multiple roots have probability zero).

Choosing randomly a
polynomial in $\mathcal{P}_d$ is equivalent to choosing randomly and
independently its roots. For a positive number $r$, the
probability $p_{d,r}$ that $d$ uniformly and independently
distributed points in the unit disk have mutual distance at least $r$
is at least
\[
p_{d,r}\geq (1 - r^2)(1-2r^2)\ldots (1-(d-1)r^2)
\]
(the unit disk has area $\pi$, and after $k$ roots are selected, the
$k+1$-st root must avoid an area of at most $k\pi r^2$; this has
probability $(\pi-\pi kr^2)/\pi=1-kr^2)$.

Since $\log(1+x) \ge x/(1+x)$ for $x>-1$, we get
\[
\log p_{d,r}\ge \sum_{k=1}^{d-1} \log(1-k r^2) \ge \sum_{k = 1}^{d-1}
\frac{-k r^2}{1-k r^2}
\ge -r^2 \frac{\sum_{k = 1}^{d-1} k}{1-d r^2} \ge -r^2
\frac{d^2/2}{1-d r^2} \ge
-d^2 r^2
\;,
\]
where the second and last inequalities hold if $d r^2 < 1/2.$ Hence
\[
p_{d, r}\ge
\exp(-d^2 r^2)\ge 1 - d^2 r^2
\;.
\]
If $r = 1/d^{1+\eta}$ (which implies $d r^2 < 1/2$),
then $p_{d,r}\ge
1-1/d^{2\eta}$ and thus the claim follows.
\end{proof}

We combine the previous two lemmas in the following claim.

\begin{lemma}\label{lemma: probabilistic conditions}
There exists a universal constant $C>0$ with the following properties. Fix $\eta \in (0,1/2]$. Then for each degree $d\in\mathbb{N}$, the following holds with probability  at least $ 1-2 d^{-2\eta}$: if a polynomial with roots $\alpha_1, \ldots, \alpha_d$ is randomly chosen in $\mathcal{P}_d$ or $\mathcal{P}_d/\Sigma_d$, then we have simultaneously
\begin{description}
     \item[Area Condition (AC)]  every disk in $\C$ with area $A$ contains at most
$k(A)$ points among the roots $\alpha_1, \ldots, \alpha_d$ with

          $$k(A) = \left\{ \begin{array}{ll}
              C d\log d\cdot A & \mbox{if $A \geq 1/d$};\\
            C \log d & \mbox{otherwise}.\end{array} \right.$$
     \item[Distance Condition (DC)] the mutual distance between any pair
of roots is at least ${1}/{d^{1+\eta}}$.
\end{description}

\end{lemma}
\begin{proof}
     We are interested in $P(AC = \texttt{true} \mbox{ and } DC =
\texttt{true})$, which equals
\[
1 - P(AC = \texttt{false} \ \ \hbox{or}\ \    DC = \texttt{false})
\geq  1- P(AC = \texttt{false}) - P(DC = \texttt{false}).
\]
By Lemma \ref{lemma:disk} we have $P(AC = \texttt{false}) \leq 1/d$,
and by Lemma \ref{lemma: mutual distance} $P(DC = \texttt{false})\le
1/d^{2\eta}$. Hence the claim follows.
\end{proof}

\subsection{Proof of the main theorem}\label{sec: proof}

Recall that $p(z)=(z-\alpha_1)\cdots(z-\alpha_d)$ is a complex polynomial of degree $d$ (from $\mathcal{P}_d$ or $\mathcal{P}_d/\Sigma_d$) with a simple root $\alpha$ and $(z_n)_{n\geq 1}$ is a sequence of iterations (under the Newton map $N_p$) that converges to $\alpha$ (see the beginning of Section \ref{sec: result}). In this section we will use the two conditions \textbf{AC} and \textbf{DC} to prove Theorem \ref{Thm:Main}. While \textbf{DC} guarantees that proximity to a root implies fast convergence (Lemma \ref{lem: quadratic convergence}), \textbf{AC} gives a lower bound on the displacements along an orbit far away from the roots. More
precisely, we have the following statement. 

\begin{lemma}\label{lemma: main}
     Suppose that the polynomial $p$ is such that the Area Condition in Lemma~\ref{lemma: probabilistic conditions} holds for some constant $C$. If $z_n\in S_K \cap \mathbb{D}_2(0)$ for some $K\in \mathbb{Z}$, then
\[
|z_n - z_{n+1}|\geq \frac{1}{(1 + 2C\log d)2^{K+1} + 16\pi Cd\log d }
\;.
\]
If $z_n \not \in \mathbb{D}_2(0)$, then $|z_n - z_{n+1}| > 1/d$.
\end{lemma}
\begin{proof}
     The fact that $z_n\in S_K$ means that the closest root, say  
$\beta$, is at distance ${c}/{2^K}$ for some $c\in (0.5, 1]$, and all the
other roots satisfy $
     |z_n-\alpha_j|\geq {c}/{2^K}$. First suppose that $z_n\in S_K \cap
\mathbb{D}_2(0)$. This implies that $K\geq -2$. Let $T_k := \{z\in
\mathbb{C}:
2^{-k-1} < |z-z_n| \leq 2^{-k}\}$ for $k = -2, \ldots, K$. Then all the
roots are contained in $\bigcup_{k = -2}^{K}T_k$. The Area
Condition implies that the number of roots in $T_k$ is bounded by $\pi C  d\log d\cdot 4^{-k}$ for
$\pi 4^{-k}\geq 1/d$, and by $C\log d$ otherwise. Thus we have
\begin{align}
\left|\sum_{j=1}^{d}\frac{1}{z_n-\alpha_j}\right|
&\leq \left|\frac{1}{z_n-\beta}\right| + \sum_{\alpha_j \not=
\beta}\left|\frac{1}{z_n-\alpha_j}\right| = \frac{2^K}{c} + \sum_{k
= -2}^{K}\sum_{\substack{\alpha_j \not= \beta\\\alpha_j \in
T_k}}\left|\frac{1}{z_n-\alpha_j}\right| \nonumber
\\
      &\leq  \frac{2^K}{c} + \sum_{k = -2}^{\lfloor 0.5\log_2
\pi d\rfloor}\sum_{\substack{\alpha_j \not= \beta\\\alpha_j \in
T_k}}\left|\frac{1}{z_n-\alpha_j}\right| +
        \sum_{k = 1+\lfloor 0.5\log_2\pi d\rfloor}^{K}\sum_{\substack{\alpha_j \not=
\beta\\\alpha_j \in T_k}}\left|\frac{1}{z_n-\alpha_j}\right|
\nonumber
\\
      &\leq  \frac{2^K}{c} + \sum_{k = -2}^{\lfloor 0.5\log_2 \pi d\rfloor}
\pi C d\log d\cdot 4^{-k}2^{k+1} + \sum_{k = 1+\lfloor 0.5\log_2\pi
d\rfloor}^{K} C\log d\cdot  2^{k+1} \nonumber
\\
      &\leq  2^{K+1} + 16\pi Cd\log d + C\log d\cdot 2^{K+2}\;. \nonumber
\end{align}

Therefore
\[
|z_n - z_{n+1}| = \frac{1}{\left|\sum_{j=1}^{d}\frac{1}{z_n-\alpha_j}\right|}\geq
\frac{1}{(1 + 2C\log d)2^{K+1} + 16\pi Cd\log d } \;.
\]

For the case $z_n \not \in \mathbb{D}_2(0)$ we have
\[
|z_n - z_{n+1}|^{-1} = \left|\sum_{j=1}^{d}\frac{1}{z_n-\alpha_j}\right| <
\sum_{\alpha_j}1 = d \;,\]
and so $|z_n - z_{n+1}| > 1/d$.
\end{proof}

\begin{corollary}\label{cor: displacement}
Suppose that the polynomial $p$ of degree $d$ satisfies the Area Condition in Lemma \ref{lemma: probabilistic conditions} with some constant $C>0$. 
Then there exists a constant $C_{\text{disp}}>0$ depending only on $C$ such that, whenever $z_n \in S_k$ for some $n\in \mathbb{N}$ and $k\in \mathbb{Z}$, we have
\begin{enumerate}
     \item if $2^{-k} \geq 1/d$, then $|z_n - z_{n+1}| \geq
\frac{C_{\text{disp}}}{d\log d}$.
     \item if $2^{-k} < 1/d$, then $|z_n - z_{n+1}|
\geq\frac{C_{\text{disp}}}{k2^k }$.
\end{enumerate}
\end{corollary}
\begin{proof}
   For $z_n\in \mathbb{D}_2(0)$, Lemma \ref{lemma: main} gives
\[
   |z_n - z_{n+1}| \geq \frac{1}{(1 + 2C\log d)2^{k+1} + 16\pi Cd\log d } \;.
\]
   If $2^{-k} \geq 1/d$, i.e.,\ $2^{k+1} \leq 2d$, the
denominator is at most $Cd\log d(4+16\pi)(1+o(1/d))$,
so the displacement
is at least $\frac{C'}{d\log d}(1+o(1/d))$ for some constant $C'$ depending only on $C$. 

On the other hand, if $2^{-k} < 1/d$, i.e.\ $d < 2^{k}$, the denominator is at most $C\log d\cdot 2^k(4+16\pi)(1+o(1/d))=C k 2^k \log 2 (4+16\pi)(1+o(1/d))
$, so the
displacement is at least $\frac{C''}{k2^k}(1+o(1/d))$ for some universal constant $C''$. 

Finally, $z_n\not \in \mathbb{D}_2(0)$ implies $k < -1$ and
Lemma \ref{lemma: main} gives $|z_n - z_{n+1}| > 1/d$. 

This implies the existence of a constant $C_\text{disp}$ as claimed; its value depends only on $C$; for large $d$ it approaches the value $1/C(4+16\pi)$.
\end{proof}

In order to estimate the required number of Newton iterations, we will need two complementary lemmas: one that assures quadratic convergence near the roots, and another one that implies definite use of area, and thus an upper bound for the number of iterations, when we are far from the roots.

\begin{lemma}[Quadratic Convergence]
\label{lem: quadratic convergence}
If, for fixed $\eta>0$, the mutual distance between any two roots of $p$ is at least $1/d^{1+\eta}$ and $z_n\in S_k$ with $2^{-k} < {1}/{8d^{2+\eta}}$, then
the orbit of $z_n$
converges to the closest root $\alpha$, and $\log_2|\log_2\varepsilon
- 5|$ iterations of $z_n$ are sufficient to get $\varepsilon$-close
to $\alpha$.
\end{lemma}
\marginpar{I still don't really understand how the $\eta$ gets lost; it seems to happen in this lemma.}

\begin{proof}
Indeed, if $z_n\in S_k$ and $\alpha$ is the closest root to $z_n$,
then $|z_n - \alpha| < {1}/{8d^{2+\eta}}$ and for every root
$\alpha_j\neq\alpha$ we have 
\[
|z_n-\alpha_j| \geq |\alpha - \alpha_j| - |\alpha - z_n| >
1/ d^{1+\eta}-1/8d^{2+\eta} \ge
(8d+1)/8d^{2+\eta} > (4d + 3)|z_n - \alpha|
\;.
\]
Therefore by \cite[Lemma~16]{D-draft}, we need no more than
$\log_2|\log_2\varepsilon - 5|$ iterations to get $\varepsilon$-close
to $\alpha$.
\end{proof}

For the second lemma, let $\varphi\colon U\to \mathbb{D}$ be a Riemann
map with $\varphi(\alpha) = 0$. If $|\varphi(z_n)| < 1/\sqrt 2$ (``region of
fast convergence''), then we are in the region of quadratic convergence and
according to \cite[Lemma~11]{D-draft} starting at $z_n$ we
need only $1+\log_2|\log_2 \varepsilon - 5|$ iterations to get
$\varepsilon$-close to the root $\alpha$. 
However, if $|\varphi(z_n)|$ is larger, we have the following lemma (note that $0.707\approx 1/\sqrt 2>e^{1/2}-1\approx 0.649)$. It essentially says that the hyperbolic geodesic within $U$ connecting $z_n$ to $z_{n+1}$ has a definite neighborhood (that we call a ``thick curve'') that is still contained in $U$ and that uses up a definite amount of area within $U\cap D_{3R/2+1}(0)$. These ``thick curves'' are essentially disjoint, and this limits the number of possible orbit points.

\begin{lemma}\label{lemma: area bound}
For every $n$ with $|\varphi(z_n)| > e^{1/2} - 1$, there are open
connected subsets $V_n\subset
U \cap D_{3R/2+1/2}(0)$ with $z_n,z_{n+1}\in\overline{V_n}$ and $\mbox{area}(V_n)\ge
|z_n-z_{n+1}|^2/2\tau,$ having the following property: whenever $n$
and $m$ are such that $\min \{|\varphi(z_n)|, |\varphi(z_m)|\} >
e^{1/2} - 1$ and $|n-m|\ge \lceil 2\tau+6\rceil$, we have  $V_n\cap
V_m=\emptyset$.
\end{lemma}
\begin{proof}
Let $\gamma\colon[0,s]\to U$ be the hyperbolic geodesic within $U$
connecting $z_n$ to $z_{n+1}$. For each $z=\gamma(t)$, let $\eta(t)$
be the Euclidean distance from $\gamma(t)$ to $\partial U$, and let
$X_t$ be the straight line segment (without endpoints) perpendicular
to $\gamma(t)$ of Euclidean length $\eta(t)$, centered at
$\gamma(t)$. Let $V_n:=\bigcup_{t\in(0,s)}X_t$. Then all $V_n$ are
open and connected with $z_n,z_{n+1}\in\overline{V_n}$, and the Euclidean area of
$V_n$ is at least $|z_n-z_{n+1}|^2/2\tau$: this follows as in
\cite[Lemma~9]{D-draft} (in this reference, the areas restricted to
certain domains $S_k$ are calculated; omitting this restriction, we
obtain the result we need, and the computations only get simpler). Moreover,
the orbit $(z_n)$ is $R$-central, so it is contained in $D_R(0)$ together with the hyperbolic geodesic segments connecting consecutive orbit points. Since the unit disk contains other
roots than $\alpha$, we have $\eta(t)<R+1$ along $\gamma([0,s])$, so $V_n\subset D_{3R/2+1/2}$ (each point in $X_t$ has distance less than $\eta(t)/2$ from $\gamma(t)$).

The fact that $V_n\cap V_m$ are disjoint when $|n-m|>2\tau+6$ is
proved in \cite[Lemma~12]{D-draft} (again for restricted domains, but
this is immaterial for the proof).
\end{proof}

The final step towards proving our main result is in the next theorem. Let $\eta \in (0,1/2]$ be fixed.

\begin{theorem}\label{thm:one root}
      Let the polynomial $p$ be randomly chosen in $\mathcal{P}_d$ or
$\mathcal{P}_d/\Sigma_d$ and let $(z_n)_{n\geq 1}$ be an $R$-central orbit converging
to a root $\alpha$ with $d_U(z_1, z_2)\leq \tau$ for $\tau < 2\log
d$. Then with probability at least $1-2d^{-2\eta}$, the required number of
iterations for $z_1$ to get
$\eps$-close to $\alpha$ is
\[
O\left(d^2\log^4 d R^2 + \log|\log\eps - 5\right|) \;.
\]
\end{theorem}

\begin{proof}
By Lemma~
\ref{lemma: probabilistic conditions}, there exists a universal constant $C>0$ such that the conditions \textbf{AC} and
\textbf{DC} hold with probability at least $1-2d^{-2\eta}$. By Corollary \ref{cor: displacement}, there is a universal constant $C_{\text{disp}}$ (depending only on $C$) and we have lower bounds on the displacement along the orbit. 
Choose $M$ so that $ 2^M - 1>3R/2+1$. We distinguish the following three cases.
\begin{description}
\item[The Far Case] we have $z_n\in S_k$ with $2^{-k}\ge 1/d$ and $|\varphi(z_n)|>e^{1/2}-1$ . By
Corollary \ref{cor: displacement} (1)
we have  $|z_n - z_{n+1}|
\geq\frac{C_{\text{disp}}}{d\log d }$.  By Lemma~\ref{lemma: area bound}, any Newton iteration $z_n\mapsto
z_{n+1}$ with $z_n\in S_k$ needs area at least
\[
\frac{|z_n-z_{n+1}|^2}{2\tau}
\ge
     \frac{C^2_{\text{disp}}}{2\tau d^2\log^2 d} \;.
\]
Moreover, the pieces of area for the
iterations $z_n\mapsto z_{n+1}$ and $z_{n'}\mapsto z_{n'+1}$ are
disjoint provided that $n-n' \geq 2\tau + 6$, and all these pieces of
area are contained in the disk $D_{3R/2+1}(0)$ with $R$ universally
bounded.

The total number of such iterations $D_{3R/2+1}(0)$ can accommodate is
thus at most
\[
C'd^2(\log d)^2 \tau \lceil 2\tau + 6\rceil \]
for a universal constant $C'$.

     \item[The Intermediate Case]  we have $z_n\in S_k$ with
$1/8d^{2+\eta}\leq 2^{-k} < 1/d$ and $|\varphi(z_n)|>e^{1/2}-1$ . Then $\log_2 d < k \leq 3 +
(2+\eta) \log_2 d$. By Corollary \ref{cor: displacement} (2) we have
$|z_n - z_{n+1}| \geq{C_{\text{disp}}}/{k2^k }$.
Thus by \cite[ Proposition 13]{D-draft}, the set $S_k$ contains at most
{\allowdisplaybreaks
\begin{align}
      &\pi d\left(2^{-k+1} + \frac{C_{\text{disp}}}{k2^k}\right)^2 \left(2\tau +
2^{k-1}\frac{C_{\text{disp}}}{k2^k}\right)\lceil 2\tau + 6\rceil
\frac{k^22^{2k}}{C^2_{\text{disp}}} \nonumber \\
      &= \pi d 2^{-2k}k^{-2}(2k + C_{\text{disp}})^2 \left(2\tau +
\frac{C_{\text{disp}}}{2k}\right)\lceil 2\tau + 6\rceil \frac{k^22^{2k}}{C^2_{\text{disp}}}
\nonumber \\
      &= \pi d C^{-2}_{\text{disp}}(2k + C_{\text{disp}})^2 \left(2\tau + \frac{C_{\text{disp}}}{2k}\right)\lceil
2\tau + 6\rceil                    \nonumber \\
      &\leq \pi d C^{-2}_{\text{disp}}(6  + (4+2\eta)\log_2 d + C_{\text{disp}})^2 \left(2\tau +
\frac{C_{\text{disp}}}{2\log_2 d}\right)\lceil 2\tau + 6\rceil           \nonumber
\\
      &\leq C''d \log^2 d (2\tau + 1)\lceil 2\tau + 6\rceil
\nonumber
\end{align}
}%
orbit points for some universal constant $C''$. 
There are
$3+(1+\eta)\log d$ possible values of $k$ in the Intermediate Case,
so $\bigcup_{k} S_k$ (for all $k$ in the Intermediate Case) can
accommodate at most
\[
(1+\eta)C''d \log^3 d (2\tau + 1)\lceil 2\tau + 6\rceil
\]
orbit points for some universal constant $C''$.
\item[The Near Case]
we have $z_n \in S_k$ with $2^{-k} < {1}/{8d^{2+\eta}}$ or $|\varphi(z_n)|\le e^{1/2}-1$. In the first case, Lemma~\ref{lem: quadratic convergence} applies so $\alpha$ is the closest root to $z_n$ and we
need $\log_2|\log_2\varepsilon - 5|$ iterations to get
$\varepsilon$-close to it. In the second case, we already observed after Lemma~\ref{lem: quadratic convergence} that at most $1+\log_2|\log_2\varepsilon - 5|$ iterations are required for the same conclusion.

\hide{Thus it only remains to consider the orbit points $z_n\in S_k$
for $k \leq 3 + (2+\eta)\log_2 d$; these are not necessarily in the
region of ``fast`` convergence (their image under the Riemann map
$\varphi$ as constructed in \cite[Section 3.3]{D-draft} has absolute
value at least $e^{1/2} - 1$). Since the orbit $(z_n)$ is contained
in ${D}_{2^M - 1}(0)$ by hypothesis, we have $z_n\in S_k$ with $k
\geq -M$ for all $n\geq 0$.
}
\end{description}
Since $\tau$ is not more than $O(\log d)$ and the Far Case dominates the Intermediate
Case, the claim follows.
\end{proof}

We now conclude the main statement.
\begin{proof}[Proof of Theorem \ref{Thm:Main}]
     By Theorem \ref{thm: set of starting points}, for each root the set $\mathcal S_d$  contains
a starting point satisfying the conditions of Theorem~\ref{thm:one root}. In
particular, these orbits are $R$-central for a universally bounded
value of $R$. Note that the $d$ roots have to compete for the
available area in $D_{3R/2+1/2}(0)$ (the sets $V_n$ from Lemma~\ref{lemma: area bound} are contained in the immediate basins). Since the estimates in the proof of
Theorem \ref{thm:one root} are based on the area (except for the
Near Case where the orbit gets to the region of quadratic
convergence), we get the same estimate for the combined number of
iterations (except that the estimate $\log |\log \varepsilon|$
applies for each root separately, thus it is multiplied by $d$).
\end{proof}

\begin{remark}
\label{Rem:Optimal}
This result is close to optimal in the sense that the power of $d$
cannot be reduced for any set of starting points that is
bounded away from the unit disk. The reason is that outside the unit
disk $N_p$ is conjugate to the linear map $w \mapsto \frac{d-1}{d} w$
by  \cite[Lemma~4]{HSS}, so at least $O(d)$ iterations are required for
each ``good'' starting point to get close to the unit disk where the
roots are located, and at least $O(d^2)$ for all the $d$ starting
points combined.
\end{remark}

\noindent
\emph{Acknowledgement}. We would like to thank Victor Pan, Michael Stoll and
Stanislav Harizanov for interesting and encouraging discussions. We are also grateful to the anonymous referee and the editor for their valuable comments that led to improvements of our manuscript. Finally, we would also like to thank the Deutsche Forschungsgemeinschaft for their support.

\end{document}